\documentclass[article,noinfoline]{imsart}
\usepackage{graphicx,enumerate}

\usepackage{jr2,comment}

\newtheorem{theorem}{Theorem}
\newtheorem{lemma}{Lemma}

\newtheorem{corollary}{Corollary}
\newtheorem{example}{Example}
\startlocaldefs
\def\bb#1{\mathbb{#1}}

\def\boldsymbol#1{\setbox\ewb\hbox{$#1$}%
    \setlength{\deno}{-\wd\ewb+0.05em}{ #1}\hspace{\deno}{#1}}
\endlocaldefs

\begin{document}

\baselineskip=22pt

\begin{frontmatter}

\title{Generalized maximum likelihood estimation of the mean of parameters of mixtures. With applications
to sampling and to observational studies.}

\author{\fnms{Eitan} \snm{Greenshtein}\ead[label=e2]{eitan.greenshtein@gmail.com}}
\address{Israel Census Bureau of Statistics; \printead{e2}}
\affiliation{Israel Census Bureau of Statistics}

\author{\fnms{Ya'acov} \snm{Ritov}
\ead[label=e3]{yaacov.ritov@gmail.com}}
\address{University of Michigan; \printead{e3}}
\affiliation{University of Michigan.}

\runauthor{ Greenshtein,   Ritov}

\maketitle

\begin{abstract}
Let $f(y\mid \theta), \; \theta \in \Omega$ be a parametric family, $\eta(\theta)$ a given function, and $G$ an unknown mixing
distribution. It is desired to estimate $E_G  (\eta(\theta))\equiv \eta_G$ based on independent observations $Y_1,...,Y_n$, where $Y_i \sim f(y\mid \theta_i)$, and $\theta_i \sim G$ are iid.

We explore the Generalized Maximum Likelihood Estimators (GMLE) for this problem. Some basic properties and representations
 of those
estimators are shown. In particular we suggest a new perspective,  of the weak convergence result by Kiefer and Wolfowitz (1956), with implications to a corresponding setup in which $\theta_1,...,\theta_n$ are {\it fixed} parameters.
We also relate the above problem, of estimating $\eta_G$, to nonparametric empirical Bayes estimation under a squared loss.

Applications of   GMLE  to  sampling problems are presented.
The performance of  the GMLE is demonstrated both in simulations and through a real data example.


\end{abstract}

\end{frontmatter}

\section{Introduction}

Let $f_\theta(y) \equiv f(y\mid \theta), \; \theta \in \Omega, $ be a parametric family
of densities with respect to some measure $\mu$.
Let $G$ be an unknown mixture distribution. We observe
$Y_1,...,Y_n$, which are realizations of the following process.
Let  $\theta_i \sim G$, $i=1,...,n,$ be independent, where  conditional on $\theta_1,...,\theta_n$, $Y_i \sim f(y\mid \theta_i)$, $i=1,\ldots,n$, are independent.  Given a function $\eta( \theta)$, let $\eta_G\equiv\eta(G)\equiv E_G \eta(\theta)$. Based on our observations, it is desired to estimate
$\eta_G$.


Consider for example the case where $Y_i \sim N(\theta_i,1)$ and $\theta_i \sim G$ where $G$ is supported away from 0. If $\eta(\theta)=E_\theta(Y)$, then the natural estimator
for $\eta_G$ is $\frac{1}{n} \sum Y_i$, which is consistent and efficient. On the other hand, deriving a consistent estimator for $\eta_G$ is
less obvious if $\eta(\theta)={1}/{\theta}$. Yet, the Generalized Maximum Likelihood Estimator (GMLE) defined below, see \eqref{eq:hatG}, yields a simple consistent estimator for both functionals: If   $\hat{G}$ is a  GMLE estimator for $G$ a GMLE estimator of $\eta_G \equiv E_G \eta(\theta)$ is defined by
\begin{equation} \hat{\eta}_G \equiv E_{\hat{G}} \eta(\theta) \end{equation}

\noindent{\bf GMLE for $\boldsymbol{G}$.}
Given a distribution $G$ and a dominated family of distributions with densities {\nolinebreak$\{ f(y\mid \theta):\; \theta \in \Omega \}$}, define
$$ f_G(y)= \int f(y\mid \vartheta) dG(\vartheta).$$

Given observations $Y_1,...,Y_n$,
a GMLE $\hat{G}$ for $G$ (Kiefer and Wolfowitz, 1956) is defined  as:
\begin{equation} \label{eq:hatG} \hat{G}= \argmax_{{G}} \; \Pi f_{{G}}(Y_i); \end{equation}
the maximization is with respect to all probability distributions.
Note, GMLE for $G$ is also referred  in the literature as Nonparametric Maximum Likelihood
Estimator (NPMLE). We use GMLE as a tribute to its originators,   Kiefer and Wolfowitz, (1956).
Traditionally, a GMLE estimator $\hat{G}$ for a mixture $G$ is approximated via EM algorithm on a finite subset of the parameter space, see the seminal paper, Laird (1978).
Koenker and Mizera (2014) suggested
exploitation of convex optimization techniques.

Much of the research on GMLE is about characterization of the GMLE $\hat{G}$ in terms of the size of its support, see, e.g.,
Lindsay (1995). Asymptotics and efficiency of the estimator $\hat{\eta}_G$ for $\eta_G$ is studied in semiparametric theory, see e.g.,
Bickel, et.a.l (1995).
One goal of this paper is to present some new problems and applications where GMLE estimators are useful:
we elaborate in particular on examples of `sampling with no response', `post-stratification', and `observational studies'.

We also relate GMLE estimators $\hat{\eta}_G$ to nonparametric empirical Bayes and compound decision problems, and provide useful and insightful representations.
In nonparametric empirical Bayes the goal is to estimate  the specific values $\eta_i \equiv \eta(\theta_i)$, $i=1,...,n$,
under a given loss, mainly, as in our case, a squared loss. A mixing distribution $G$ induces a joint distribution of $(Y,\theta)$. We denote the corresponding conditional expectation of
$\theta$ conditional on
$Y$,  by $E_G(\theta\mid Y)$.
The nonparametric empirical Bayes approach  is to estimate
$\eta_i$, $i=1,...,n$ by the plug-in estimator
$\hat{\eta}_i=E_{\hat{G}}(\eta(\theta)\mid Y_i)$, see, the general approach in
the seminal papers of Robbins (1951, 1956, 1965), see the plug-in approach in,
e.g., Jiang and Zhang (2009), Koenker and Mizera (2014), and $g$-modeling in Efron(2014).  In subsection \ref{sec:eqre}
we show that
$\hat{\eta}_G$ may be represented as $\hat{\eta}_G=\frac{1}{n} \sum \hat{\eta}_i= \frac{1}{n} \sum E_{\hat{G}} (\eta(\theta) \mid  Y_i)$.
In the
case of $Y_i \sim N(\theta_i,1)$ with $\eta(\theta) = E_\theta Y$, $\theta_i \sim G$ are iid (and more generally)
we show, that $\hat{\eta}_G  \equiv E_{\hat{G}} \eta(\theta) = \frac{1}{n} \sum Y_i$.

Asymptotics of GMLE  estimators $\hat{\eta}_G$, mainly in terms of consistency, is explored both theoretically
and through simulations.
There are cases where  $\hat{G}$ is not unique, e.g., in situations  where $G$ is non-identifiable.
When $\hat{G}$ is not unique,
we consider as GMLE  any  $\hat{\eta}_G= E_{\hat{G}} \eta(\theta)$ that corresponds to any GMLE $\hat{G}$.
In such cases, as $n \rightarrow \infty$, often,  not every sequence of GMLE is consistent; we argue that GMLE estimators are still plausible and worthwhile, and suggest
``GMLE related'' Confidence-Intervals for  $\eta_G$.

Given  realizations $Y_1,...,Y_n$ as above, a  related problem, studied in Zhang (2005), is the prediction
of the realized $\sum \eta(\theta_i)$, where $\theta_i \sim G$ are iid.
His approach is also related to GMLE.
Zhang elaborates on the difference between the
estimation and the prediction
problems in terms of efficiency.
He also presents interesting applications where it is desired to estimate the realized $\sum \eta(\theta_i)$.


%


Our  problem may be formalized without considering the parameters, $\theta_i$s, as random and appealing to a mixture distribution $G$.
We consider  $\theta_1,...,\theta_n$ as unknown parameters and  $Y_i \sim f(y\mid \theta_i)$,  $i=1,...,n$, as independent (but not identically distributed)  observations. The goal is  to estimate $\sum \eta(\theta_i)$. Let $G^n$ be the empirical distribution of $\theta_1,...,\theta_n$. Let $\hat{G}^n$ be the GMLE based on $Y_1,...,Y_n$,
as defined in (\ref{eq:hatG}) (i.e., the GMLE pretending $\theta_1,\dots,\theta_n$ are \iid random variables sampled from a distribution $G$). Under
suitable triangular array formulation and conditions, the sequence  of signed  measures $(\hat{G}^n - G^n)$, converges weakly to the zero-measure.
However, a reference to such a general result is not known to us, and we give a proof in
Section \ref{sec:weak}. This weak convergence result motivates
estimating
 $\frac{1}{n}\sum \eta(\theta_i)= E_{G^n} \eta(\theta) =\eta_{G^n}$ by $ \;  E_{\hat{G}^n} \eta(\theta).$

As in most  of the literature, we often appeal to a mixing distribution $G$ and random  $\theta_i \sim G$, which makes
the formulation more convenient.
The two approaches of appealing to a mixing distribution $G$, versus avoiding it,
are analogous to empirical Bayes versus compound decision approaches.
See, e.g.,  Zhang (1997), Brown and Greenshtein (2009), and the fore mentioned
seminal papers by Robbins. It is also related to the distinction between incidental and  random nuisance
parameters, see, e.g., Pfanzagl (1993).

In Section 2 we present some motivating examples. In Section 3 we present some theoretical
results concerning representations of $\hat{\eta}_G$, and some asymptotics. In Section 4
we present simulations. In Section 5
we sketch a `GMLE related' Confidence Interval for $\eta_G$. In Section 6 we present a real data example taken from the Israeli Social Survey.

\section{Examples. \label{sec:Ex}}

A main motivation for our study of GMLE are the following sampling models and  problems in the context of stratification, and observational studies.

In stratified sampling  and in post-stratification, it is desired to have a very fine stratification.
Then, conditional on a fine stratum it appeals that missing observations are missing completely at random. Similarly in observational study (or, in convenience sampling), fine stratification makes the assumption that conditional
on the strata the success or failure of a treatment are independent  of the, often unknown, allocation
mechanism to treatment and control. However, under very fine natural stratification many strata  are likely
to be randomly empty, i.e., with no sampled observations.  In practice, when the number of observations is high the stratification used is fine to the level of existence of a meaningful number of empty cells. The GMLE method is a natural  way of handling
the resulting difficulties of empty strata.

We elaborate on the above.
Suppose that it is desired to  estimate the proportion of unemployed in the population.
A sampled  subject will respond to the survey with unknown probability. The proportion of unemployed among the responders in the survey is strictly a biased estimator of their proportion in the population, since there is a strong observed correlation between employment status and response.  It is known that in the relevant surveys, subgroups  with higher unemployment rates have lower response rates. We may hope, however, that if we consider small enough and homogenous stratum, the willingness to answer and the answer will be, conditionally on the stratum, practically independent. Thus, the missing responses would be practically missed completely at random (MCAR) within each stratum. See, e.g., Little and  Rubin (2002).

Let $K_i$ be the number of full observations in Stratum $i$ (i.e., $K_i$ were sampled and responded). Let $X_i$ be the number of unemployed among the observations from Stratum $i$.
For simplicity, we consider a situation of $n$ strata
with equal weights. Our observations are $Y_i=(X_i,K_i)$,  $i=1,\dots,n$, where the conditional distribution of
$X_i$
conditional on $K_i$ is $B(K_i,p_i)$.
It is desired to estimate the population's proportion $p=n^{-1}\sum p_i$.

Similar considerations apply in observational studies. Suppose that conditional on a fine stratification, within each stratum
the probability of success or failure, under treatment or control, is independent  of the unknown allocation mechanism to treatment or control. Given a treatment, we thinks of
$X_i$, the number of successes of the treatment in stratum $i$, and $K_i$, the number of
times the treatment was applied in Stratum $i$. We observe $Y_i=(X_i,K_i)$. One approach
of analyzing such observational  data is to pair
observations based on their propensity score. Using our method we do not need  pairing, and
we may handle  even the extreme case, where  treatment was applied in one subset of strata, while control was
applied in another, disjoint, subset of strata.


\subsection{Sampling Models}

We consider the following  two realistic  scenarios,
where the sample size $K_i$ is random.  In light of Section \ref{sec:weak}, the following motivating models apply  to both setups of random and fixed
$\theta_1,...,\theta_n$. We use the notations and formulation of random $\theta_i$, $i=1,...,n$, for convenience.

Model (i) \emph{ Stratified sampling with non-response.}
A random
sample of  $\kappa_i$  subjects from stratum $i$ is sampled. The probability of a random  subject from stratum $i$ to
respond is $\pi_i \leq 1$.
Thus, the number of actual responses, $K_i$,
$K_i \sim B(\kappa_i, \pi_i)$ is a random variable.

Model (ii): \emph{  Post-Stratification}.
We assume that $K_i$ has a  $Poisson(\lambda_i)$ distribution, $i=1,2,\dots,n$.  This scenario is reasonable whenever we have a convenience sample,
observational studies, or, very low response rate.

In Model (i),  $\theta_i=(\theta_{i1},\theta_{i2}) \equiv (\pi_i,p_i)$,
where $\pi_i$ is the probability of response, while $p_i$ is the
proportion  (of, say, unemployed)  in stratum $i$.
Conditional on $\theta_i$, $X_i\mid K_i \sim B(K_i,p_i)$ and $K_i \sim B(\kappa_i, \pi_i)$. In Model (ii),  $\theta_i=(\theta_{i1},\theta_{i2}) \equiv (\lambda_i, p_i)$.
Given $\theta_i$, $X_i\mid K_i \sim B(K_i,p_i)$  as before,
and $K_i \sim Poisson(\lambda_i)$.

Suppose we are interested in estimating $\eta_G=E_G \eta(\theta)$, for $\eta(\theta_i)=p_i$.
If there are no empty
strata, i.e., $K_i>0$, $i=1,...,n$ then the obvious and naive estimator:
 $$\hat{p}_N=\frac{1}{n} \sum_{i=1}^n  \frac{X_i}{K_i},$$ is applicable.
The problem with
the above estimator is that many strata may be empty, i.e., with corresponding $K_i=0$.
When there are some empty strata, a common ad-hoc approach is of `collapsing strata' where after
the data is observed some empty strata are unified with  non-empty ones. An extreme collapsing is to a single stratum, which yields the  Extreme Collapsing estimator:

$$\hat p_{EC}=\frac{ \sum X_i}{\sum K_i},$$
which is, in fact, desirable when $p_1=...=p_n$ but not in general.

Another  way of handling the difficulty of empty strata, is to assume that strata are `missing at random', that is $K_i$ is uncorrelated
with $p_i$, and strata with $K_i=0$ may be simply  ignored.

Our approach, of applying GMLE,
handles cases with many empty strata in a natural way, which is not ad-hoc and does not rely heavily on missing at random type of assumptions.

Finding a GMLE for a two dimensional distribution $G$,
outside of the normal setup, only recently appears
in the literature, see,
e.g., Gu and Koenker (2017), Feng and Dicker (2018). One reason might be its recent popularity
due to Koenker and Mizera`s computational methods. Our Model (ii) is presented in Feng and Dicker, however, their motivation and context is very different than ours,
and is not related to sampling. In addition, they study the estimation of the
the individual $p_i$, $i=1,...,n$ via non parametric empirical Bayes, while we study the estimation of $\eta_G=E_{G} \eta(\theta)$, where $\eta(\theta_i)=p_i$. We also provide consistency results in the estimation of $\eta_G$.

Different related sampling models are
studied in Greenshtein and Itskov (2018).
In their censored-case, $n$ individuals are attempted to be interviewed,
where there are at most $\kappa^0$ attempts for each individual $i, \; i=1,...,n$. The probability of a response from individual $i$ in
any single attempt is $\pi_i$, and attempts are independent. Let $K_i \leq \kappa^0$ be the number of interviewing attempts of individual $i$, then $K_i$ is a truncated geometric variable. Let $Z_i$ be a 0-1 valued r.v., e.g., indicator of unemployment,
$p_i=P(Z_i=1)$. Suppose $\theta_i \equiv (\pi_i,p_i)$ are i.i.d distributed $G$.
It is desired to estimate $E_G Z_i = E_G(\eta(\theta))$, where $\eta(\theta)=p$.
The possible  outcomes, of  the attempted interviews of individual $i$, are
$(Z_i,K_i) \in \{ (1,K_i), \; (0,K_i),\; 1\le K_i \leq \kappa^0,\; \& \; NULL\}$, where
$NULL$ indicates non-response in $\kappa^0$ attempts.
This scenario is within our setup as one may think of individuals as strata of size one.

The truncated case of Greenshtein and Itskov (2018)
is similar to the above censored case, except that we do not know about the existence of items that did not respond.  Denote by $G^t$ the conditional distribution of $\theta$, conditional on
response. Then: $dG^t(\theta) \propto P_\theta (K_i \leq \kappa^0) dG(\theta)$. Let $\hat{G}^t$ be the GMLE for $G^t$, which is based only on the
(non-truncated) observations $\{i:\;  K_i \leq \kappa^0 \}$. Then,
when $\pi_i>0$ w.p.1,  we suggest the GMLE estimator
$$\hat{\eta}= \int \eta(\theta) d\hat{G}(\theta) \equiv
\frac{\int \eta(\theta)  P_\theta^{-1}(K_i \leq \kappa^0)   d\hat{G}^t(\theta)}{\int P_\theta^{-1}(K_i \leq \kappa^0) d\hat{G}^t(\theta)}.$$

\section{Asymptotic Results and equivalent representations of GMLE estimators.}
\label{sec:theor}


We consider mainly the model in which $\theta_i=(\lambda_i,p_i)$, $\theta_i \sim G$,  $i=1,...n$ are iid random variables and consider the estimation of the function $\eta(\theta_i)=p_i$.
In light of subsection \ref{sec:weak} below, analogous results may be derived when  $\theta_1,...,\theta_n$ are considered unknown parameters.

\subsection{Consistency and Asymptotics in Models (i) and (ii).}

We will now argue that under  (ii) if  $P(\lambda=0)=0$, then the expected value of $p$ can be consistently estimated, while under model (i) it cannot. The difference between the two models is that on the one hand, in model (ii) $K_i$ may get the values $0,1,2,\dots$, and as $n\to\infty$, the number of observed outcomes grows to infinity. On the other hand, in model (i), $\kappa_i$ is bounded by some finite $\kappa$, then the total number of outcomes is $(\kappa+1)(\kappa+2)/2$, which is not enough for the identification of the distribution or even the mean of of $p_i$. Here are the details.
	
\subsubsection{ Consistency in Model (ii)  \label{sec:II}}

In Model (ii)   we have consistency in the estimation of $\eta_G$,
when $P_G(\lambda=0) \equiv G(\{ \lambda=0 \})=0$,
as proved in the following theorem.


It is convenient to re-parametrize
the problem, as follows. Given $X_i \sim B(K_i, p_i)$ conditional on $K_i$,
and $K_i \sim Poisson(\lambda_i)$.
Denote $W_{i1} \equiv X_i$, and $W_{i2}=K_i- X_i$. Then $W_{i1}$ and
$W_{i2}$ are  independent  Poissons conditional on $(\lambda_i,p_i)$, with corresponding parameters
$\xi_{1i} \equiv p_i \lambda_i$
and $\xi_{2i} \equiv (1-p_i)\lambda_i$.

Before presenting the theorem, recall that $G$ is identifiable in the model $f_G(y)=\int f_\vartheta(y) dG(\vartheta)$ if $f_G=f_{\tilde{G}}$ implies $G=\tilde{G}$.

\begin{theorem}
Let $G_\xi$ be the distribution of $(\xi_1,\xi_2)$ and  let $\hat{G}_\xi$ be the GMLE based on iid $(W_{i1}, W_{i2}), \; i=1,...,n$.  then $G_\xi$ is identifiable and  $\hat{G}_\xi$ converges weakly to $G_\xi$.
\end{theorem} \label{thm:con}
\begin{proof}
The proof follows from Kiefer and Wolfowitz (1956). Checking the conditions is standard,
the identifiability condition is verified, e.g., by
Karlis and Xekalaki (2005).
\end{proof}

\begin {corollary}
 for any  function $\eta(\theta)$, such that $\eta(\theta)=\psi(\xi_1,\xi_2)$,
for  $\psi$ which is continuous  and bounded on the support of $(\xi_1,\xi_2)$ under $G$,
$E_{\hat{G}} \eta(\theta) \rightarrow E_G \eta(\theta)$. In particular, if under $G_\xi$, $P(\lambda{>}0)=1$, then $E_G p $ is identifiable.
\end{corollary}

The condition $P(\lambda{>}0)=1$ is necessary, since otherwise $P(\eta{=}0)>0$ while the distribution of $p$ conditioned on $\lambda=0$ is unidentified, since there are no observations on $p$ from the atom at $\{\lambda=0\}$. On the other hand, if  $P(\lambda{>}0)=1$ than  $\eta_\epsilon(\theta)=\psi_\epsilon(\xi_1,\xi_2)
\equiv\frac{\xi_1}{\xi_1 + \xi_2+\eps} 1\{\xi_1+\xi_2{>}0\}$ is bounded and continuous, $E_{\hat{G}}
\eta_\eps(\theta) \rightarrow E_G \eta_\eps(\theta)$ for any $\epsilon>0$. But  $\lim_{\epsilon \rightarrow 0} E_G \eta_\epsilon= E_G \eta(\theta)$, and hence $E_{\hat{G}}
\eta(\theta) \rightarrow E_G \eta(\theta)$.



 We conclude that when confining   to the class of distributions in
$\Gamma= \{ G \mid  G(\{\lambda=0 \})=0 \}$,
we have consistency of the estimator $E_{\hat{G}} \eta(\theta)$ for $E_G \eta(\theta)$ for
any $G \in \Gamma$.
However, the convergence may be arbitrarily slow and depends heavily on the probability concentration of $G$ in the neighborhood of 0.

Consider the following example. $\lambda \sim G_\lambda$, $p=p_0 +\delta 1(\lambda<\lambda_0)$ for some $\lambda_0$. To get a bound on the rate, suppose we know all parameters except for $\delta$ and we are told when $\lambda<\lambda_0$. Then the only sample relevant to the information is the sample of size $O_p\bigl(nG_\lambda(\lambda_0)\bigr)$ coming from $\lambda<\lambda_0$. Given the sample, $\sum X_i 1(\lambda<\lambda_0)$ is binomial, hence we can estimate $\delta$ with accuracy of $\bigl(n\lambda_0G_\lambda(\lambda_0)\bigr)^{-1/2}$. This error has a contribution of
$G_\lambda(\lambda_0)\bigl(n\lambda_0G_\lambda(\lambda_0)\bigr)^{-1/2}$. Since $\lambda_0$ can be arbitrarily small  $G_\lambda(\lambda)$ can converge to 0 as slow as we want, the estimating error of $\delta$ has as slow rate as we want.

\subsubsection{ Inconsistency in Model (i) }

Again, we consider the function $\eta(\theta_i)=p_i$. Now $\theta_i=(\pi_i,p_i)$.

In Model (i) there is no consistency in estimating $G$ for every $G$, neither a consistency in
estimating $E_G(\eta(\theta))$. This is a result of the non-identifiability, and it is demonstrated in
the following example for   $\kappa_i \equiv 1$.

\begin{example} \label{ex:bin}
When $\kappa_i \equiv 1$ there are $M=3$ possible outcomes of the $i'th$ observation, we list them as: $X_i=1$,  $X_i=0$ (while $K_i=1$), and $K_i=0$; the corresponding probabilities are:
$(\pi_i p_i, \pi_i(1-p_i), (1-\pi_i))$.
Suppose the outcomes of $n$ realization have $n_1$ occurrences of $X_i=1$,  $n_2$ occurrences of $X_i=0$, and $n_3$ occurrences of $K_i=0$. Note $(n_1,n_2,n_3)$ is multinomial.
Suppose $\frac{1}{n}(n_1,n_2,n_3)=(0.25,0.25,0.5)$, obviously the following $\hat{G}_1$ and
$\hat{G}_2$
are both GMLE. Let $\hat{G}_1$ be degenerate at $(\pi,p)=(0.5,0.5)$. Let $\hat{G}_2$ be the distribution whose
support is $(0,1), (1,0), (1,1)$ with corresponding probabilities 0.5, 0.25, 0.25.
Then, obviously both $\hat{G}_1$ and $\hat{G}_2$ are GMLE, while $E_{\hat{G}_1} \eta(\theta) =0.5 \neq E_{\hat{G}_2} \eta(\theta) = 0.75.$

\end{example}

More generally, if $\kappa_i\in \sck$, \sck a bounded set, then $(\kappa_i,K_i,X_i )$ may get at most $M$ values, and hence the distribution $G$ is not identified (as long as it is not constraint to have a finite support with cardinality less the above number of possible values). In fact, one can estimate the expectations
\begin{equation}\label{estFun}
g_G(k,x)=E_G \pi^k (1-\pi)^{\kappa_i-k}p^x(1-p)^{k-x},\quad 0\le x\le k\le \kappa_i\in \sck .
\end{equation}
Note that $\eta(\pi,p)\equiv p$ is not the linear span of the functions under the expectation in right hand side of  \eqref{estFun}. Consider now the system of equations:
\eqsplit[estFunPlus]{
    g_{G_0}(k,x)
    &=E_G \pi^k(1-\pi)^{\kappa_i-k}p^x(1-p)^{k-x},\quad 0\le x\le k\le \kappa_i\in\sck
    \\
    \eta&= E_G p
}
The set of equations \eqref{estFunPlus} is a linear system of $M$, say, linearly independent equation in $G$. In fact, fix any $\kappa_i$ and $k$, the span of the subset of functions with $\kappa_i$ and $k$ is $\pi^k(1-\pi)^{\kappa_i-k}p^i$, $i=0,1,\dots,k$. Similarly $\fun{span}\{\pi^k(1-\pi)^{\kappa_i-k}:k=1,\dots,\kappa_i\}= \fun{span}\{\pi^j:j=1,\dots,\kappa_i\}$. Thus,  span of the functions on the top line of \eqref{estFunPlus} is span$\{\pi^jp^i: i\le j\le\max\sck\}$. In particular, $p$ is not in this span, and hence \eqref{estFunPlus} is a set of linearly independent equations.  Suppose, for simplicity, that the true distribution $G_0$ has $M$ support points.  Then, for any $\eta$, \eqref{estFunPlus} has as a  solution with the same support as $G_0$. It is a distribution function, since 1 is in the span of the functions in the right hand side of the top line of \eqref{estFunPlus} and hence $\int dG=\int dG_0=1$. However, it is not necessarily positive. But, by the inverse function theorem, the solution is continuous in $\eta$ and hence for any $\eta$ in some neighborhood of $\eta_0=E_{G_0}p$ there is a positive solution $G$ which is a probability distribution function.

This argument fails for model (ii),  where $\sck=\{0,1,2,\dots\}$, since $p$ is now in the closed linear span of the probability functions. However, this is another explanation for the slow potential convergence. We need to observe the rare large values of $\kappa_i$ to estimate the functional.



\subsection{ The GMLE and the mean of the empirical Bayes estimates}
\label{sec:eqre}

Given a function $\eta(\theta) $, suppose it is desired to estimate
$\eta_G=E_G \eta(\theta)$. Again, for convenience we consider the notations of random parameters $\theta_1,...,\theta_n$,
but, in light of Section \ref{sec:weak} the results apply also to the setup of fixed $\theta_1,...,\theta_n$.
By the  following theorem, the estimator $\hat{\eta}_G= E_{\hat{G}} \eta( \theta)$ equals to the average of
$E_{\hat{G}}(\eta(\theta) \mid    Y_i)$, $i=1,...,n$. The appeal of this fact is that if we estimate
 the values of the individual $\eta(\theta_i)$ via nonparametric empirical Bayes under squared loss, specifically by
$E_{\hat{G}} (\eta(\theta )\mid  Y_i)$, there is a consistency and agreement between the estimates of the individual parameters and the estimate of their total, or, of their average.
In Zhang (2005), the problem of estimating random sums involving a latent variable is explored. One approach in
Zhang (2005) is to estimate the random sum, by the sum of the estimated  conditional expectations of the summands. This is analogous to estimate  $\sum_i \eta(\theta_i)$ by
$\sum_i E_{\hat{G}} (\eta(\theta) \mid  Y_i)$, the last term equals to $nE_{\hat{G}} \eta(\theta)$ by the following theorem.

\begin{theorem} \label{thm:agreement}
Assume $\eta(\theta)$ is a bounded function, then
\begin{equation}\label{eq:thm:agreement}
\hat{\eta}_G \equiv E_{\hat{G}} \eta(\theta) = \frac{1}{n}\sum _i E_{\hat{G}}(\eta(\theta )\mid  Y_i).
\end{equation}
\end{theorem}

\begin{proof}


Let  $$d\hat{G}_t(\theta)= \bigl(1+t(\eta(\theta)-\hat\eta_G)\bigr)d\hat{G}(\theta).$$
Since $\eta(\theta)$ is bounded, it follows that for $\eps>0$ small enough, $\scg=\{ G_t, \; t \in (-\eps,\eps) \}$ is a curve of  cdf's (i.e., $G_t$ is a positive measure with total mass 1).

Since $\hat{G}$ is a GMLE, it maximizes the likelihood within $\scg$. Hence:
\begin{eqnarray*}
0&=& \frac1n\frac{d}{dt} \sum_i  \log (\int f(Y_i\mid \theta) d\hat{G}_t(\theta) ) \mid _{t=0}\\
&=&\frac1n\sum_i \frac { \int \bigl(\eta(\theta)-\hat\eta_G\bigr) f(Y_i\mid \theta) d\hat{G}(\theta) }
{\int f(Y_i\mid \theta) d\hat{G}(\theta)}=\frac1n\sum_i \int \eta(\theta) d\hat{G}(\theta\mid Y_i) - \hat{\eta}_G.
\end{eqnarray*}
This concludes the proof of the theorem.

\end{proof}

Theorem \ref{thm:agreement} represent the GMLE estimator, $\hat\eta_G$ as an average of what may look like of almost \iid random variables, $n^{-1}\sum E_G\bigl(\eta(\theta)\mid Y_i\bigr).$ The latter is  $\sqrt n$ consistent and asymptotically normal. However, this is misleading. All terms in the sum in \eqref{eq:thm:agreement} depend on all other through $\hat G$ and $ E_{\hat{G}}(\eta(\theta )\mid  Y_i)$ is possibly a biased estimator for $ E_{{G}}(\eta(\theta )\mid  Y_i)$. In fact,    as we argue,  $\eta_{\hat G}(\theta)$ may be inconsistent or consistent but converges in a very slow rate. On the other hand, in many other models,  the GMLE is actually consistent, see the discussion in Section 7.8 of Bickel et al. (1995). One such a case  is described in Section \ref{sec:expfam} below. See further discussion there.

\subsubsection  { GMLE for the mean of exponential  mixtures }
\label{sec:expfam}

Let $Y_i \sim N(\theta_i,1), \; i=1,...,n$ be independent observations, the obvious estimator for  $\sum \theta_i$ is $\sum Y_i$.
One may wonder about a comparison between the trivial estimator $\sum Y_i$ and the estimator $nE_{\hat{G}} \theta=\sum E_{\hat{G}} (\theta\mid  Y_i)$.
The   nonparametric empirical Bayes $E_{\hat{G}} (\theta\mid  Y_i)$ is strictly better than $Y_i$ as a point-wise estimator of
$\theta_i$. c.f., Brown and Greenshtein (2009), Jiang and Zhang (2009), and Koenker and Mizera  (2014).
Similarly one may wonder about the similar model, estimating $\sum \lambda_i$ when $Y_i \sim Poisson(\lambda_i)$.
In the following we will show that the
two estimators are in fact equal. Thus, GMLE  does not improve over the trivial estimator, yet, it does not  harm.

In the following theorem the notations are for a one dimensional
exponential family, but it applies for a general exponential family.

\begin{theorem} \label{thm:expfam} Let $Y_1,...,Y_n$, be independent observations,
$Y_i \sim f_{\theta_i} (y)$.
Suppose that the density of $Y$ is of the form $f_\theta(y) \equiv f(y\mid \theta)= \exp(\theta y-\psi(\theta))$, $\theta \in \Omega$,
with respect to some dominating measure $\mu$.
Let $\hat{G}$ be a GMLE supported  on the interior  of the parameter set $\Omega$.
Let $\eta(\theta)= E_{\theta} Y$.
Then: $$nE_{\hat{G}} \eta(\theta)= \sum Y_i.$$
\end{theorem}
\begin{proof}
Given the observations $Y_1,...,Y_n$, let $\hat{G}$ be a GMLE. Define
the translation $\hat{G}_\Delta(\; . \;)=\hat{G}(\; .  - \Delta)$.

Note: $$ \int \exp(\theta y -\psi(\theta)) d \hat{G}_\Delta(\theta)=
 \int \exp((\theta-\Delta) y -\psi(\theta-\Delta)) d \hat{G}(\theta).$$

Since $\hat{G}$ is GMLE,
\begin{eqnarray*}
0&=& \frac{d}{d\Delta} \sum \log(\int \exp(\theta Y_i -\psi(\theta))
d \hat{G}_\Delta(\theta) ) \mid _{\Delta=0} \\
&=&  \frac{d}{d\Delta}
\sum \log ( \int \exp((\theta-\Delta) Y_i -\psi(\theta-\Delta) ) d \hat{G}(\theta) ) |_{\Delta=0}\\
&=& \sum  \frac{  \int (-Y_i + \psi'(\theta) )  \exp(\theta Y_i -\psi(\theta)) d \hat{G}(\theta) }
{\int \exp(\theta Y_i -\psi(\theta)) d \hat{G}(\theta)}\\
&=& -\sum Y_i + \sum E_{\hat{G}} (\psi'(\theta) \mid  Y_i)\\
&=& -\sum Y_i + nE_{\hat{G}} \eta(\theta)
\end{eqnarray*}

The last equality follows by Theorem  \ref{thm:agreement},
and since  for exponential family
$\eta(\theta)=\psi'(\theta)=E_\theta Y$.

\end{proof}

The above theorem implies that under an exponential family setup,
$E_{\hat{G}} \eta(\theta)$ is unique  for $\eta(\theta)=E_\theta Y$ even if $\hat{G}$ is not. See the following example.
\begin{example}
Let  $Y_i \sim Bernoulli(p_i)$ be independent where $p_i \sim G$  are independent, $i=1,...,n$.
Suppose $Y_1=...=Y_{\frac{n}{2}}=0$ and  $Y_{\frac{n}{2}+1}=...=Y_n=1$.
Then, obviously both $\hat{G}_1$ and $\hat{G}_2$ are GMLE, where $\hat{G}_1$ has half of its mass
at $p=0.1$ and the other half at $p=0.9$, while $\hat{G}_2$ is degenerate at
$p=0.5$. Obviously $\hat{G}_1 \neq \hat{G}_2$,   however, $E_{\hat{G}_1} p= E_{\hat{G}_2} p=0.5$.
\end{example}

In the normal model, $Y_i \sim N(\theta_i,1)$ all finite cumulants  of $\theta$ can be estimated in the $\sqrt n$ rate by the corresponding cumulants  of $Y$. For example,
estimate $E_G\theta$ by $n^{-1}\sum Y_i$. That these estimators are equivalent to the GMLE can be argued by a similar argument as in the proof of Theorem \ref{thm:expfam} by taking further derivatives. However, other functions cannot be estimated in such a rate, e.g., the deconvolution problem of estimating $G(\theta_0)=E_G1(\theta\le\theta_0)$ has a logarithmic rate see Fan (1991). Even a very smooth function like $\eta(\theta)=e^{-\alpha^2\theta^2/2}$, $\alpha\ge 1$, cannot be estimated in the parametric rate, see Donoho and Low (1992) . However, estimating functions of finite cumulants is efficient, as the tangent space is saturated, and any mean is an efficient estimate of its expectations. An example of a less trivial function that can be estimated efficiently at the $\sqrt n$  is the density at 0 of $G*N(0,\alpha^{-2})$, $\alpha^2<1$, which in our framework is $E_G\eta(\theta)$, where $\eta(\theta)=\sqrt{\alpha^2/2\pi}e^{-\alpha^2 \theta^2/2}$. This is also, the density at 0 of $f_G*N\bigr(0,(1-\alpha^2)/\alpha^2\bigr)$, and hence can be estimated with the unbiased kernel estimator:  $n^{-1}\sqrt{(1-\alpha^2)/2\pi\alpha^2}\sum e^{-\alpha^2Y_i/2(1-\alpha^2)}$. See Bickel et al. Section 4.5 for the tangent space and for the argument that it is efficient.

\subsection{ The compound decision model and weak convergence} \label{sec:weak}

The seminal paper of Kiefer and Wolfowitz (1956), appeals to  a completely unknown distribution $G$, assuming that
$\theta_i \sim G$, $i=1,2,...$ are iid.  They prove weak convergence of $\hat{G}$ to $G$.
It may be seen in the introduction of that paper,
that the  authors are somewhat uncomfortable with this (very weak) assumption.
The weak convergence in KW (1956) is explored from a broader perspective in Chen, J. H. (2017).

In this subsection we  attempt to present a weak convergence result when $\theta_1,...,\theta_n$ are considered as fixed unknown parameters that may depend on $n$. That is, we consider a general triangular array  where at stage $n$ the parameters are some $(\theta_1^n,...,\theta_n^n)$, as also elaborated in the sequel.

Recently there are results about rates of convergence of $\hat{G}$ and other estimators to $G$. See, e.g.,
Saha and  Guntuboyina (2020) for the multivariate Gaussian case,
Philippe and Kahn (2018) for finite mixtures. Those papers do not cover the
weak convergence for triangular arrays
presented in this subsection. In  Philippe and Kahn (2018)  a non-GMLE estimator, $\tilde{G}$ for $G$, is studied and they provide minmax rates for
the Wasserstein distance between $\tilde{G}$ and $G$, for mixtures with $m$ components. However, those rates do not imply our weak convergence  result in the triangular array formulation.
Specifically, in our triangular array setup, as the number of observations $n$ increases,
the number of components $m$ in the mixture is also increased since we allow  $m=n$.

Given any triangular array sequence of sets $\{\theta_1^n,...,\theta_n^n \} $, $n=1,2,...$, denote by $G^n$ its corresponding empirical distribution. Given independent observations $Y_1^n,..., Y_n^n$, $Y_i^n \sim f(y\mid \theta_i^n)$, we denote the corresponding GMLE, as defined above, by $\hat{G}^n$. That is, $\hat G^n$ is the GMLE ignoring the fact that the observations are not \iid from the mixture. In the sequel we may omit the superscript $n$ and write $\theta_i$,  $Y_i$.

Note that, in our current setup the true joint likelihood $f_G^n(Y_1,...,Y_n)$
does {\it not} satisfy $f_G^n(Y_1,...,Y_n)= \Pi_i f_G(Y_i)$, since $Y_i$ are {\it not} independent.

\begin{theorem} \label{thm:weak} Assume A1-A5 below, then   for any bounded and continuous function $f$, $\int f d\hat{G}^n- \int f d G^n\to_{a.s.} 0$.


\end{theorem}

{\bf Assumptions}
\begin{description}

\item{A1} The parameter space $\Omega$ is compact.

Let $$h(y)=\sup_\theta |\log(f(y\mid \theta))|, $$

\item{A2}   $\sup_{\theta \in \Omega} E_\theta h^4(Y)<\infty$.





\item{A3}
The functions $f(y_0\mid \theta)$
are uniformly bounded and continuous in $\theta$  for every $y_0$.



\item{A4} The class of densities $\{ f_G(y) \}$, that correspond to the set of all possible mixtures $\{G\}$,
is identifiable.


\item{A5} For every $G_0 \in \{G\}$, every $\epsilon>0$, and every $M>0$, the set
$$ \{ G \; \vert \; \Vert \log(f_G^{(M)})-\log(f_{G_0}^{(M)})\Vert_\infty <\epsilon\},$$
is an open set under the weak convergence topology, where $f_G^{(M)}(Y)=[f_G(Y)]_{M}$ and $[x]_M=\max(-M,\min(X,M)).$
\end{description}

Condition A1 is used to avoid situations where mass can escape. The theorem fails when $(\theta_1,\dots,\theta_n)=(n,n+1/n,\dots,n+(n-1)/n)$ since we can consider a bounded and continuous $f$ which is not uniform continuous, for example, $f(\theta)=\sin(\theta^4)$. This may be considered as a technical issue, as the GMLE is translation equivariant, and in this case we do have convergence to uniform, as could be obtained by restricting to a smaller set of test functions, e.g., to bounded functions with a bounded first derivative. More problematic example is with the parameters at sample size $n$ being $(1,2,\dots,n)$. In this case the convergence fails even when $f(\theta)=e^{-\theta^2}$. Assumption A1, however, can be replaced by a weaker one that ensures compactness of the measures such as strong moments conditions. We preferred to keep it simple, but strong enough for our models (although it does not cover, for example, a Gaussian mixing distribution).

The purpose of A5 is to ensure, that given an open cover for $\{ G\}$, a finite covering  subset may be extracted by compactness. The assumption is implied  by A1--A4 if, in addition, it is assumed that for every $M>0$, the functions $\log((f_\theta^{(M)}(y)) $, $\theta \in \Omega$ are  uniformly continuous.

\begin{proof}

By A1 we trivially have tightness, and hence any sequence $G^n$ of measures,
has a weakly converging sub-sequence. Our plan is to show that
every sub-sequence  of  $a_n\equiv \int f d\hat{G}^n- \int f d G^n$ has a further sub-sequence that converges to zero w.p.1.
The later implies the assertion
of the theorem.

Given any sub-sequence of $a_{n}$, take a  corresponding further sub-sequence $n_j$,
so
that  $G^{n_j} \Rightarrow_w G^0$ and $\hat{G}^{n_j} \Rightarrow_w  \hat{G}^0$, for some
$\hat{G}^0$ and $G^0$.

Denote by $G^n_1$ the joint distribution
of $\theta_1,...,\theta_n$
 when the parameters are sampled
independently from $G^n$, as if the parameters are sampled with replacement from the set $\{\theta_1,...,\theta_n\}$. We denote by $G^n_2$ the joint distribution of $\theta_1,...,\theta_n$
when the sampling is done without replacement.
Note that, for any function $\psi$,  \begin{equation} \label{eq:eq} E_{G_1^n} \sum_i \psi(Y_i)= E_{G_2^n} \sum_i \psi(Y_i). \end{equation}


We keep the notations that for any function $\psi$,   $E_G \psi(Y)$ is the expectation when $\theta \sim G$ and  $Y\mid\theta \sim f(y \mid  \theta)$.

By Lemma \ref{lem:lem} in the Appendix,
under both sequences of joint measures $G^{n_j}_1$ and $G^{n_j}_2$,
\begin{equation} \label{eq:main}
\frac{1}{n_j} \sum_{i=1}^{n_j}  \log (f_{\hat{G}^{n_j}} (Y_i))  \rightarrow_{a.s.} E_{G^0}  \log (f_{\hat{G}^0}(Y)).
\end{equation}
\noindent Similarly,    under both sequences of joint measures $G^{n_j}_2$, and $G^{n_j}_1$
\begin{equation}
\frac{1}{n_j} \sum_{i=1}^{n_j}  \log ( f_{{G}^{n_j}} (Y_i) ) \rightarrow_{a.s.}  E_{G^0}   \log(f_{ G^0} (Y)).
\end{equation}

By definition of $\hat{G}^{n_j}$ as GMLE, $ \frac{1}{n_j} \sum_i \log(f_{\hat{G}^{n_j}} (Y_i)) \geq
 \frac{1}{n_j}  \sum_i \log( f_{ {G}^{n_j} }(Y_i) )$,
thus under the sequence of measures $G^{n_j}_2$:
\begin{equation}
\begin{split}
E_{G^0}  \log (f_{\hat{G}^0}(Y)) &=  \lim  \frac{1}{n_j} \sum_{i=1}^{n_j} \log (f_{\hat{G}^{n_j}} (Y_i))
\\
&\geq \lim \frac{1}{n_j}\sum_{i=1}^{n_j}  \log (f_{G^{n_j}}(Y_i))
\\
&= E_{G^0}   \log(f_{ G^0} (Y)) ,
\end{split}
\end{equation}

Since $f_{G^0}=\argmax_{f_G } E_{G ^0} \log( f_G(Y)) $,
we obtain that:
\begin{equation}
E_{G^0} \log (f_{\hat{G}^0} (Y)) \leq E_{G^0} \log (f_{{G}^0}(Y)).
\end{equation}

From the last two inequalities
we obtain:
 $$E_{G^0} \log (f_{G^0} (Y)) = E_{G^0} \log (f_{\hat{G}^0}(Y)).$$

By the identifiability assumption A4, and by concavity of the $\log$ function,
$\argmax_{\{ f_G \}}  E_{G^0} \log (f_G(Y)) $   is unique, and  $\hat{G}^0=G^0$.

This concludes our proof.

\end{proof}

{\bf Remark}.  In {\it compound decision}, where $\theta_1,...,\theta_n$ are fixed, it is desired to estimate $\eta_i=\eta(\theta_i)$.
 Then, the goal  is to approximate the optimal separable estimator, some times also  termed
optimal simple symmetric  estimator. Under a squared loss, the later estimator for $\eta_i$, given an observation $Y_i$, is
$ E_{G^n}(\eta(\theta_i)\mid Y_i)$, see, e.g., Zhang (1997), Brown and Greenshtein (2009).
In our triangular array setup, our weak convergence result, implies  under suitable conditions
 $[E_{G^n}(\eta(\theta_i)\mid Y_i)- E_{\hat{G}^n}(\eta(\theta_i)\mid Y_i)] \rightarrow 0$. A simple such possible additional condition in our triangular array setup, is
that $\liminf f_{G^n}(Y_i^n)> 0$.

\section{Simulations.}

In this section,
we simulate the estimation of $\eta_G=E_G \eta(\theta)$, where $\eta(\theta_i)=p_i$, under both models (i) and (ii).
In all of the following simulations, for convenience,
$\eta_G$ equals 0.5.
For any parameter configuration, the number of
repetitions is 50.

We compute the GMLE $\hat{G}$ via EM algorithm on a
grid.
As suggested by Koenker and Mizera(2014), we search for a ``confined GMLE''
where we confine our search to distributions with
a finite and specific support.
Our choice of a specific support is ad-hoc, see some rigorous treatment in Dicker and
Zhao (2014). The search for GMLE among distributions  on the specific grid via EM-algorithm,
is also the approach in
Feng and Dicker (2018), see further elaboration there.

The grids for $(\theta_{i1},\theta_{i2})$ contain
$40 \times 40=1600$ equally spaced grid points in a range that fits the relevant problem.
The parametrization in Model (ii) is via a two dimensional Poisson, as explained in
Section \ref{sec:II}. The EM algorithm started with $G$ uniform on the grid and ran for 1000 iterations.

\subsection{ Poisson sample sizes.}

Table \ref{tab:discrete} presents the results of Model (ii) with 2-points support distribution $G$.  There are 500 strata corresponding to each of the two points in the support of $G$.
The table presents  the mean and the sample standard deviation of the naive  and the GMLE
estimators for $\eta_G$.

\begin{table}
\caption{Poisson Simulation discrete $G$. The estimated mean and standard deviations of two estimators. }
\label{tab:discrete}
\begin{center}
\begin{tabular}{ |c||c|c| }
 \hline
    Support points of $(\lambda,p))$ & The naive estimator & The GMLE \\
	\hline \hline	
 $(2,0.4),(1,0.6)$ & {\bf 0.486}, (0.014) & {\bf 0.503}, (0.020) \\
 $(2,0.2),(1,0.8)$  & {\bf 0.453}, (0.015) & {\bf 0.496}, (0.018) \\
 $(2,0.2),(0.5,0.8)$ & {\bf 0.385}, (0.013) & {\bf 0.505}, (0.022) \\
 \hline
\end{tabular}
\end{center}
\end{table}

In Table \ref{tab:cont} we report on one such set of simulations.
In the three cases presented in Table \ref{tab:cont} there are again 500 strata  of two types. The probability $p^I$ of first 500 strata
 is fixed while for the rest of the strata the probability is $p^{II}=(1-p^I)$. The Poisson parameter $\lambda$ is continuous in all the three cases and is chosen from a uniform distribution, $U(0.5,1)$ in the first 500 strata and from $U(0.5,2)$ in each of the remaining 500 strata.

\begin{table}
\caption{Poisson simulation with continuous $G$. The distribution of $\lambda$ is
$U(0.5,1)$, for the first 500 strata, and  $U(0.5,2)$ for the remaining strata. The binary probabilities $p^I$ corresponding to the first  500 strata are given the table. The binary probabilities of in the rest of the strata are $1-p^I$. The table gives estimates of the mean and standard deviation of the two estimators. }
\label{tab:cont}
\begin{center}
\begin{tabular}{ |c||c|c| }
 \hline
     $p^I$& Naive & GMLE \\
\hline \hline	
 $0.4$ & {\bf 0.513}, (0.015) & {\bf 0.500}, (0.023) \\
 $0.3$ & {\bf 0.529}, (0.017) & {\bf 0.501}, (0.027) \\
 $0.2$ & {\bf 0.538}, (0.016) & {\bf 0.491}, (0.027) \\
 \hline
\end{tabular}
\end{center}

\end{table}

\subsection{ Binomial sample sizes.}

We next study Model (i), where $K_i$, the realized sample size from stratum $i$, is
distributed $B(\kappa_i,\pi_i)$. Again, our simulated populations have two types of strata,
500 of each type. In the simulations reported in Table \ref{tab:Binom1}, $\kappa_i=4$ for  all of the 1000 strata, and $\eta_G=0.5$ throughout the three simulations summarized in the table.  In 500 strata  $\pi_i=p_i=0.5-\del$ while in the other 500 strata, $\pi_i=p_i=0.5+\del$.

\begin{table}
\caption{The mean and standard deviation of two estimators. Binomial Simulation. $\kappa \equiv 4$ and $\pi_i=p_i=0.5\pm \del$. }\label{tab2}
\label{tab:Binom1}
\begin{center}
\begin{tabular}{ |c||c|c| }
 \hline
    $\del$ & Naive & GMLE \\
	\hline \hline	
 $0.3$ & {\bf 0.559}, (0.012) & {\bf 0.502}, (0.014) \\
 $0.2$ & {\bf 0.522}, (0.011) & {\bf 0.504}, (0.012) \\
 $0.1$ & {\bf 0.504}, (0.010) & {\bf 0.501}, (0.010) \\
 \hline
\end{tabular}
\end{center}
\end{table}

The final reported simulations are summarized in Table \ref{tab:Binomial2}.
We study Binomial sampling with various values of $\kappa$, fixed at $\kappa=1,\dots,5$. Again, there are
two types of strata, 500 strata for each of
the two types; for 500 strata $p_i$ and $\pi_i$ are sampled independently from  $ U(0.1,0.6)$, while for the rest of the strata they are \iid from $U(0.4,0.9)$.

\begin{table}
\caption{Binomial Simulations with continuous $G$. $\kappa$=1,2,3,4,5.}\label{tab3}
\label{tab:Binomial2}
\begin{center}
\begin{tabular}{ |c||c|c| }
 \hline
    $\kappa$ & Naive & GMLE \\
	\hline \hline	
 $1$ & {\bf 0.544}, (0.019) & {\bf 0.530}, (0.015) \\
 $2$ & {\bf 0.528}, (0.014) & {\bf 0.502}, (0.021) \\
 $3$ & {\bf 0.522}, (0.014) & {\bf 0.498}, (0.022) \\
 $4$ & {\bf 0.517}, (0.012) & {\bf 0.499}, (0.020) \\
 $5$ & {\bf 0.512}, (0.009) & {\bf 0.501}, (0.013) \\
 \hline
\end{tabular}
\end{center}
\end{table}

It is surprising how well the GMLE is doing
already for $\kappa=2,3$, in spite of the
non-identifiability of $G$ and the
inconsistency of the GMLE.

\begin{section} {Confidence Interval for $\eta_G$}
\label{sec:CI}
Although the GMLE $\hat{\eta}_G$ is an appealing estimator, we do not know how good is its performance,
beyond the consistency results we established. This is especially under non-identifiability where consistency is not implied.
In the following we suggest an asymptotically level-$(1-\alpha)$ conservative confidence interval  for $\eta_G=E_G \eta(\theta)$.
We only elaborate on defining the corresponding convex optimization problem. Implementing and solving the sketched
convex optimization problem, could be challenging.

Let $Z=Z(Y)$ be a random variable with $M$ possible outcomes, $z_j, \; j=1,...,M$. One may think of $M$ ``chosen cells'' in a goodness of fit test.
The considerations for the choice and for the number of cells is beyond the scope of this section, in particular $M(n) \equiv M$ is fixed. {
Then the densities $f_Y(y\mid \theta)$ induce densities $f_Z(z\mid \theta)$, $\theta \in \Omega$.} Denote
$$p_G^j=P_G(Z=z_j)= \int f_Z(z\mid \vartheta)dG(\vartheta);$$
denote $\hat{p}^j= \frac  { \{\# i \mid  Z_i=z_j\}}{n} \equiv \frac{n_j}{n}, \; j=1,...,M$, where $n_j$ is implicitly defined.
Recall that:
$$ 2( \sum_j n_j \log(\hat{p}^j)   - \sum_j  n_j \log(p_G^j ) ) \Rightarrow_G \chi^2_{(M-1)}.$$
Let $\chi^{2}_{(M-1), 1-\alpha}$, be the $(1-\alpha)$ quantile of a $\chi^2_{(M-1)}$ distribution.
Define  $$\Gamma_\alpha = \{ G \; : \;   \sum_j n_j \log({p}_G^j )\geq \sum_j n_j log(\hat{p}^j) - \frac{1}{2}\chi^{2}_{(M-1), 1-\alpha} \}.$$
Observe that $\Gamma_\alpha$ is a convex set of distributions.

The functional $\eta(G)= E_G \eta(\theta)$, is  a linear functional.

Let $$\eta^L=\min_{G \in \Gamma_\alpha} \eta(G); \; \eta^U=  \max_{G \in \Gamma_\alpha}
\eta(G).$$

By the above, finding $(\eta^L,\eta^U)$ is a convex problem, and in addition the later interval is a $(1-\alpha)$
conservative confidence interval for $E_G \eta(\theta)$.

{
The above set $\Gamma_\alpha$ is in the spirit of the F-localiztion in Ignatiadis and Wager (2021).
While our suggestion for the choice of $Z(Y)$ is ad-hoc, as often done in goodness of fit tests, they  carefully elaborate
on optimal choices of `affine estimators'. }

It may be shown that the above is also a conservative $(1-\alpha)$-level CI
for $\sum \eta(\theta_i)$ for fixed $\theta_1,...,\theta_n$, and also a $(1-\alpha)$-level
conservative credible set for $\sum \eta(\theta_i)$  when
$\theta_1,...,\theta_n$,  are independent realizations  $\theta_i \sim G$.

\end{section}

\section{   Real Data Example.}

Our example is based on data from the Social-Survey conducted
yearly by  Israel Census Bureau (ICB).
A random sample representing a $1/1000$  fraction of the age 20 or older
individuals in the registry is drawn. The  home
addresses of those in the sampled are verified and they are interviewed  in person.

We study here the  data  accumulated for
Tel-Aviv, in the surveys  collected  during
2015--2017. The total sample size  in those
three years is 1256.
There are 156 `statistical-areas' in Tel-Aviv, very
roughly of equal size, about 3000 individuals in each.
This means that around three individuals are sample from each statistical area.
Statistical-areas are considered homogeneous in many
respects, and we take them as our strata.

Let $K_i$ be the  sample size in stratum $i$, $i=1,...156$,
then our data satisfy $K_i>0$, $i=1,...,156$.
(Admittedly, we neglected a few small statistical-areas that actually
had zero sample sizes).
In our analysis, $p_i$ is the proportion of individuals in stratum $i$
that own their living place (or, it is owned by a member of their household).
The goal is to estimate $n^{-1} \sum p_i$, roughly the
proportion of individuals that own their living-place.

The naive estimator is applicable  since $K_i>0$, $i=1,...,156$. The estimated proportion   obtained by the naive estimator is:
$$\frac{1}{n} \sum_{i=1}^n  \frac{X_i}{K_i}=
{\bf  0.434}.$$
On the other hand, the `extreme collapse'
estimator, satisfy: $$\frac{\sum X_i}{\sum K_i}={\bf 0.488}.$$
The significant difference between the two estimates
has to do   {\it also} with the fact that strata are, in fact, not of
equal size. But, more  importantly for us, the
`extreme collapse'  seems  to over estimate the
proportion, since owners are over represented in
the sample. One reason is that their address in
the registry is more accurate and thus it is easier to
find them. In other words, individuals are not MCAR (missing completely at random).
This phenomena is partially corrected by the
stratification, when MAR (missing at random) conditional on the fine strata
is approximately right.


In Table \ref{tab4} we compare the GMLE and the naive estimators in  semi-real simulated  scenarios in which  only
a randomly sub-sampled with probability  $\gamma$, of the described sample, is retained. The average estimates  based on 25 simulations applied on
the real data  with $\gamma=0.1, 0.2, 0.25$, are presented
in Table 5.  The (random) number of
simulated strata with zero sample sizes, corresponding
to $\gamma=0.1, 0.2, 0.25$,  are around 70, 40, and 30,
correspondingly.

It is reasonable to assume that the number of observations in the strata are  Poisson random variables. Thus, the
sample  sizes $K_i$ in the simulated sub-sample are   ${\mbox Poisson}(\lambda_i)$ too, $i=1,...,156$.
The naive estimate based on  the entire data equals {\bf 0.434},
which is a reasonable benchmark, since that based on the entire data we have no empty strata.

 In the semi-real simulations the GMLE seems to perform much better
than the naive estimator.

\begin{table}
\caption{ Home ownership in Tel-Aviv. Performance of the two estimators as function of sampling rate. }\label{tab4}
\begin{center}
\begin{tabular}{ |c||c|c| }
 \hline
    $\gamma$ & Naive & GMLE \\
	\hline \hline	
 $0.1$ & 0.471 & 0.457 \\
 $0.2$  & 0.467 & 0.443 \\
 $0.25$ & 0.447 & 0.434 \\
 \hline
\end{tabular}
\end{center}
\end{table}

\section{Appendix}

\begin{lemma} \label{lem:lem}
Assume the setup of Section \ref{sec:weak}, Assumptions  A1--A5, and the notion of the proof of Theorem \ref{thm:weak}. Then, under both $G^{n_j}_1$ and $G^{n_j}_2$,

i) \begin{equation} \label{eq:main}
\frac{1}{n_j} \sum_{i=1}^{n_j}  \log (f_{\hat{G}^{n_j}} (Y_i))  \rightarrow_{a.s.} E_{G^0}  \log (f_{\hat{G}^0}(Y)).
\end{equation}

ii)
\begin{equation}
\frac{1}{n_j} \sum_{i=1}^{n_j}  \log ( f_{{G}^{n_j}} (Y_i) ) \rightarrow_{a.s.}  E_{G^0}   \log(f_{ G^0} (Y)).
\end{equation}
\end{lemma}

\begin{proof}

We prove part i).

 First we show that:
\begin{equation} \label{eqn:first}
\lim E_{{G}^{n_j}} \log(f_{\hat{G}^{n_j}}(Y))= E_{{G}^0} \log( f_{\hat{G}^0}(Y)).
\end{equation}
Note that by A2 for every $\epsilon>0$, there exists  large enough $m>0$, such that for every
$G', G'' \in \{ G\}$,   $$\bigl\vert E_{G'} \log f_{G''}(Y) \times I(h(Y) \leq m) - E_{G'} \log f_{G''}(Y)  \bigr\vert<\epsilon,$$
here $I$ is an indicator function.

Hence,  (\ref{eqn:first}) is implied by the following,
\begin{eqnarray*}
&&\hspace{-5em}\lim_{n_j} E_{{G}^{n_j}} \log(f_{\hat{G}^{n_j}}(Y))\times I(h(Y) \leq m) \\ &=&\lim_{n_j} E_{{G}^{0}} \log(f_{\hat{G}^{n_j}}(Y))  \times
I(h(Y) \leq m) \times \frac{ f_{{G}^{n_j}}(Y) }
{f_{{G}^0}(Y) }\\
&=&E_{G^0} \lim_{n_j} \log(f_{\hat{G}^{n_j}}(Y)) \times  I(h(Y) \leq m) \times \frac{ f_{{G}^{n_j}}(Y) }
{f_{{G}^0}(Y) } \\
&=&
 E_{{G}^0} \log( f_{\hat{G}^0}(Y))  \times  I(h(Y) \leq m).
\end{eqnarray*}
The above is by Lebesgue dominating convergence theorem, applying A2 .  Note that, the likelihood ratio $\frac{ f_{{G}^{n_j}}(Y) }
{f_{{G}^0}(Y) }$ is bounded when $h(Y) \leq m$.  Equation (\ref{eqn:first}) follows by  letting $m \rightarrow \infty$ and utilizing A2.

Under $G^{n_j}_1$, the proof will follow, by Borel Cantelli, if we show that for  every $\epsilon>0$:
\begin{equation} \label{eq:main2} \sum_{n_j} P_{{G}^{n_j}_1} \Bigl( \Bigl\vert\frac{1}{n_j}
\sum_{i=1}^{n_j} \log(f_{\hat{G}^{n_j}} (Y_i))- E_{{G}^{n_j}} \log(f_{\hat{G}^{n_j}} (Y))\Bigr\vert>\epsilon\Bigr) <\infty.\end{equation}

The complication in the above is that in the terms $f_{\hat{G}^{n_j}}  (Y_i)$, the random function $f_{\hat{G}^{n_j}}$ depends on its
argument $Y_i$.
We will prove (\ref{eq:main2}) through the following steps.

In fact, the convergence in equation (\ref{eq:main2}) is of a sub-series $n_j$, but also the
original series converges. Indeed,
in the sequel we consider the original series with  a general index $n$.
Let $\log(\Psi) \in \{ \log(f_G), \; G \in  \{G \} \}$, be a {\it fixed} function, in particular, independent of the observations. Then for every $\epsilon>0$
 $$P_{ G_1^{n} } \Bigl(\Bigl\vert\frac{1}{n} \sum_{i=1}^{n} \log(\Psi(Y_i)) -
E_{G^{n}} \log( \Psi(Y))\Bigr\vert >\epsilon\Bigr) <
\kappa/n^2,$$ for a suitable large enough $\kappa$.
The above is by is by Markov inequality, utilizing the bounded fourth moment assumption A2,
and the fact that  under ${ G_1^{n} }$ (sampling with replacement), $Y_i, \;  i=1,...,n,$ are independent.

Note, that by tightness, $\{  G\}$ is compact under the weak convergence topology, since that every sequence $G^k$ has a converging sub-sequence.
For every $\epsilon>0$ and $M>0$, by A5 and compactness of $\{G\}$, we may find  a finite cover of $L$ open sets, and corresponding functions
$\Psi_1,...,\Psi_L, \; \Psi_i \in \{ f_G \}$,
such that for every $f_0 \in \{f_G^{}\}$
\begin{equation} \label{eqn:infty}
 \inf_{\Psi \in \{ \Psi_1,...,\Psi_L \} } \Vert\log(f_0^{(M)})- \log(\Psi^{(M)})\Vert_\infty <\epsilon. \end{equation}

Next, since $L$ is finite, the above considerations coupled with Bonferroni,  imply that:
\begin{equation}
P_{{G}^{n}_1} \Bigl( \max_{\Psi \in \{\Psi_1,...,\Psi_L \} }  \Bigl\vert\frac{1}{n} \sum_{i=1}^{n} \log(\Psi(Y_i))
p- E_{{G}^{n}} \log(\Psi (Y))\Bigr\vert>\epsilon\Bigr) < \kappa/n^2,
\end{equation}
for a suitable large enough $\kappa$.

Now we show that in Equation (\ref{eq:main2}), for large enough $M$, we may neglect the terms where $|\log(f_{G^{n_j}}(Y_i))|>M$.
For every $\epsilon>0$, for large enough $M$, a.s.,
$$ \limsup \frac{1}{n}\sum_{i:\; |\log(f_{\hat{G}^{n}} (Y_i))|\geq M   }
|\log(f_{\hat{G}^{n}} (Y_i))| < \epsilon.$$
This follows since $|\log(f_{\hat{G}^{n}} (Y_i))| \leq h(Y_i)$, and by A2,
coupled with Borel Cantelli.
Hence, in the following, considering large $M$, we may neglect $$\frac{1}{n}\sum_{i:  \; |\log(f_{\hat{G}^{n}} (Y_i))|\geq M   }
|\log(f_{\hat{G}^{n}} (Y_i))|.$$

By (\ref{eqn:infty}) we may conclude, by taking large enough $M$, that for  every $\epsilon>0$,

\begin{equation} \label{eq:main4}
P_{{G}^{n}_1} \Bigl( \sup_{\Psi \in \{\ f_G\} }  \Bigl\vert\frac{1}{n} \sum_{i=1}^{n} \log(\Psi(Y_i))
- E_{{G}^{n}} \log(\Psi(Y))\Bigr\vert
>\epsilon\Bigr) < \kappa/n^2,
\end{equation} for a suitable large enough $\kappa$.

Equation (\ref{eq:main2}) follows by the last equation and the above.

The Lemma is now implied by Borel Cantelli, for the case $G^{n_j}_1$, of sampling with replacement.

The convergence under $G^{n_j}_2$ is obtained by comparing sampling with replacement to sampling without replacement. Let $S_{n_j}=\sum_{i=1}^{n_j} \log(\Psi(Y_i))$. Then,
$E_{G^{n_j}_1 } S_{n_j}^4 \geq E_{G^{n_j}_2 } S_{n_j}^4$.
It may be seen by applying Rao-Blackwell on the convex function $g(S_{n_j})= S_{n_j}^4$, when conditioning on $k_i, \; i=1,...,n_j$, where $k_i$ is the number of times $\theta_i$ was sampled
in the sampling with replacement process, that defines $G^{n_j}_1$. The implied bounded 4'th moment
of $S_{n_j}$ under $G^{n_j}_2$, coupled with (\ref{eq:eq}), implies (\ref{eq:main2}), along the lines of the above, under $G^{n_j}_2$.
This concludes our proof.
\end{proof}

\newpage

{\bf \Large References:}

\begin{list}{}{\setlength{\itemindent}{-1em}\setlength{\itemsep}{0.5em}}

\item
P.J. Bickel, C. A. Klassen, Y.Ritov and J.A. Wellner (1998).  Efficient and Adaptive Estimation For Semiparametric Models. Springer-Verlag NY.
\item
Brown, L.D. and Greenshtein, E. (2009). Non parametric
empirical Bayes and compound decision
approaches to estimation of high dimensional vector of normal
means. {\it Ann. Stat.} {\bf 37}, No 4, 1685-1704.
\item
Chen, J. (2017).  ``Consistency of the MLE under Mixture Models.'' Statist. Sci. 32 (1) 47 - 63.
\item
Dicker, L. and Zhao, S. (2014). Nonparametric empirical Bayes and maximum likelihood
estimation for high-dimensional data analysis. ArXiv preprint arXiv:1
\item
Donoho, D. L. and Low, M. G. (1992). Renormalization exponents  and optimal pointwise rates of convergence. {\it Ann. Stat.} {\bf 20}, No 2, 944-970.
\item
B. Efron. (2014). Two modeling strategies for empirical Bayes estimation. Statistical science: a
review journal of the Institute of Mathematical Statistics, 29(2):285.
\item
Fan, J. (1991). On the optimal rates of convergence for nonparametric deconvolution problems. {\it Ann. Stat.} {\bf 19} No.3, 1257-1272.
\item
Greenshtein, E. and Itskov, T (2018), Application of Non-Parametric  Empirical
Bayes to Treatment of Non-Response. {\it Statistica Sinica} 28 (2018), 2189-2208.
\item
Gu, J. and Koenker, R. (2017). Unobserved Heterogeneity in Income Dynamics: An Empirical Bayes Perspective. {\it Journal of Business and Economics Statistics}. Volume 35, 2017 - Issue 1
\item
Heinrich, Philippe, and Jonas Kahn. ”Strong identifiability and optimal minimax rates for finite mixture
estimation.” Annals of Statistics 46.6A (2018): 2844-2870.
\item
 Ignatiadis Nikolaos and Wager Stefan (2021). Confidence Intervals for Nonparametric Empirical Bayes Analysis. To appear in JASA.
\item
Jiang, W. and Zhang, C.-H., (2009), General maximum likelihood empirical Bayes estimation of normal means.
{\it Ann.Stat.}. {\bf 37} No. 4, 1647-1684.
\item
Karlis, D. and Xekalaki, E. (2005), Mixed Poisson Distributions.
{\it International Statistical Review}, 73, 1, 35–58, Printed in Wales by Cambrian Printers.
\item
Kiefer, J. and Wolfowitz, J. (1956). Consistency of the maximum likelihood estimator in the presence of infinitely many incidental parameters. {\it Ann.Math.Stat.}
27 No. 4, 887-906.
\item
Koenker, R. and Mizera, I. (2014). Convex optimization, shape constraints,
compound decisions and empirical Bayes rules. {\it JASA } 109, 674-685.
\item
Laird, N. (1978). Nonparametric maximum likelihood estimation of a mixing distribution. {\it JASA}  78, No 364, 805-811.
\item
Lindsay, B. G. (1995). Mixture Models: Theory, Geometry and Applications.
Hayward, CA, IMS.
\item
Little, R.J.A and Rubin, D.B. (2002). Statistical Analysis with Missing Data.
New York: Wiley
\item
Long, Feng and Lee, H. Dicker (2018).
Approximate nonparametric maximum likelihood for mixture models: A convex optimization approach to fitting arbitrary multivariate mixing distributions. {\it Computational Statistics \& Data Analysis} 122: 80-91.
\item
J.Pfanzagl (1993). Incidental Versus Random Nuisance Parameters. {\it Ann Stat}, 21, 1663-1691.
\item
Sujayam Saha. Adityanand Guntuboyina. "
``On the nonparametric maximum likelihood estimator for Gaussian location mixture densities with application to Gaussian denoising.'' Ann. Statist. 48 (2) 738 - 762, April 2020. https://doi.org/10.1214/19-AOS1817
\item
Robbins, H. (1951). Asymptotically subminimax solutions of compound decision problems. In Proceedings of the Second Berkeley Symposium on Mathematical Statistics and
Probability, 1950 131–148. Univ. California, Berkeley. MR0044803
\item
Robbins, H. (1956). An empirical Bayes approach to statistics. In Proc. Third Berkeley
Symp. 157–164. Univ. California Press, Berkeley. MR0084919
\item
Robbins, H. (1964). The empirical Bayes approach to statistical decision problems. Ann.
Math. Statist. 35 1–20. MR0163407
\item
Zhang, C.-H. (1997). Empirical Bayes and compound estimation of a normal mean.
Statist. Sinica 7 181–193.
\item
Zhang, C-H. (2005). Estimation of sums of random variables: Examples and information bounds. {\it Ann. Stat.} {\bf 33}, No.5. 2022-2041.

\end{list}

\end{document}